\documentclass[11pt]{amsart}
\usepackage{geometry}                
\usepackage[parfill]{parskip}    
\usepackage{amssymb, amsthm, amsmath, mathrsfs}
\usepackage{url}
\usepackage{xcolor}
\usepackage{comment}
\usepackage[all,cmtip]{xy}
\usepackage{graphicx}
\usepackage{xypic}
\usepackage{tikz}
\usepackage{mathtools}
\usepackage{hyperref}

\makeatletter
\newtheorem*{rep@theorem}{\rep@title}
\newcommand{\newreptheorem}[2]{%
\newenvironment{rep#1}[1]{%
 \def\rep@title{#2 \ref{##1}}%
 \begin{rep@theorem}}%
 {\end{rep@theorem}}}
\makeatother

\newtheorem{theorem}{Theorem}[section]
\newtheorem{lemma}[theorem]{Lemma}
\newtheorem{corollary}[theorem]{Corollary}
\newtheorem{proposition}[theorem]{Proposition}

\newtheorem{claim}{Claim}

\newreptheorem{theorem}{Proposition}

\theoremstyle{definition}
\newtheorem{definition}[theorem]{Definition}
\newtheorem{remark}[theorem]{Remark}

\title{Adjacency of three-manifolds and Brunnian links}

\author{Tye Lidman}
\address{Department of Mathematics, North Carolina State University, Raleigh, NC 27607, USA}
\email{tlid@math.ncsu.edu}

\author{Allison H. Moore}
\address{Department of Mathematics \& Applied Mathematics, Virginia Commonwealth University, 1015 Floyd Avenue, Box 842014, Richmond, VA 23284-2014, USA}
\email{moorea14@vcu.edu}

\begin{document}

\maketitle
\begin{abstract}
We introduce the notion of adjacency in three-manifolds. 
A three-manifold $Y$ is $n$-adjacent to another three-manifold $Z$ if there exists an $n$-component link in $Y$ and surgery slopes for that link such that performing Dehn surgery along any nonempty sublink yields $Z$. 
We characterize adjacencies from three-manifolds to the three-sphere, providing an analogy to Askitas and Kalfagianni's results on $n$-adjacency in knots.
\end{abstract}

\section{Introduction}

A knot $K$ is said to be $n$-adjacent to a knot $K'$ if there exists a diagram of $K$ containing a set of $n$ crossings such that changing any nonempty subset of them yields a diagram of $K'$. Any knot that is adjacent to the unknot is, of course, unknotting number one, but the condition is much more restrictive. 
For example, nontrivial knots which are $n$-adjacent to the unknot for $n\geq 3$ have trivial Alexander polynomial, are non-fibered, non-alternating, and have vanishing Vassiliev invariants of degree less than $2n-1$. 
These restrictions on their invariants are shown by Askitas-Kalfagianni \cite{AK} to result from a diagrammatic characterization of $n$-adjacent knots, $n\geq3$, as those constructed from certain spatial chord diagrams called Brunnian-Suzuki graphs \cite[Theorem 4.4]{AK}. A Suzuki graph is given by a collection of weighted, embedded arcs along an unknotted circle in $S^3$, where the arcs describe a pattern along which a sequence of bandings and clasps converts the graph into a knot. 

In this article, we generalize the concept of $n$-adjacency from knots to three-manifolds. We say that a closed, oriented three-manifold $Y$ is {\em integrally $n$-adjacent} to another three-manifold $Z$ if there exists an $n$-component link $L$ and integral multi-slope $p$ of the link such that performing Dehn surgery along any nonempty subset of $L$ yields $Z$. The triple realizing the adjacency will be denoted $(Y,L,p)$. {\em Rational adjacency}, denoted $(Y,L,\alpha)$, is defined similarly where rational surgeries are permitted.  In an analogy to Askitas-Kalfagianni, we characterize all $n$-adjacencies to the three-sphere as those arising from particular Dehn surgeries along Brunnian-like links.
This recovers Askitas-Kalfagianni's diagrammatic characterization of knots adjacent to the unknot, and we similarly obtain a statement on the finite type invariants of three-manifolds. 

Given a link in a three-manifold, and a choice of surgery slopes, the core curves of the surgery solid tori produce a new link in the surgered manifold.  We call this the {\em core} or {\em dual} of the surgery.  
Using dual links, it is in fact quite easy to construct examples of adjacent manifolds for all $n$: Let $J$ be a Brunnian link in $S^3$, and perform $(\pm 1,\ldots, \pm 1)$-surgery on $J$. This yields a homology sphere $Y$ that is $n$-adjacent to the three-sphere via the dual link. The homology sphere will be distinct from $S^3$ provided that $J$ is a non-trivial link.  For a concrete example, $(1,1,1)$-surgery on the Borromean rings produces the Poincar\'e homology sphere, which is $3$-adjacent to $S^3$, as shown in Figure \ref{fig:poincare}.  In fact, we will see shortly that all integral $n$-adjacencies for $n \geq 3$ arise from surgery on Brunnian links.

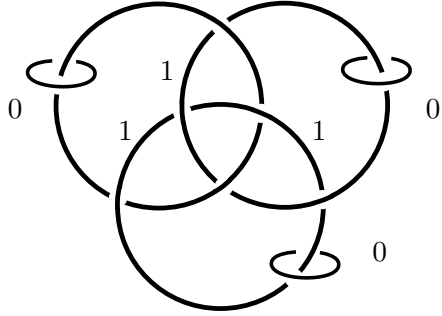
\begin{figure}
    \centering
    \begin{tikzpicture}

    \node[anchor=south west,inner sep=0] at (0,0) {\includegraphics[width=2in]{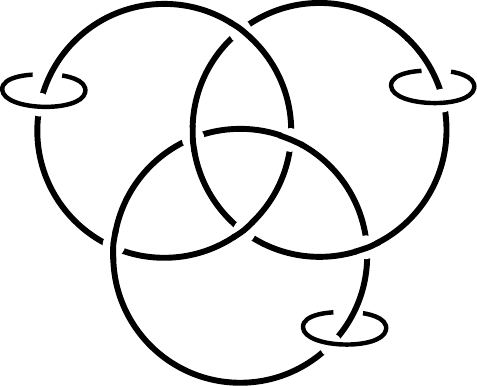}};

    \node[label=above right:{$1$}] at (1.5,2.8){};
    \node[label=above right:{$1$}] at (0.95,2){};
    \node[label=above right:{$1$}] at (3.5,2){};
    \node[label=above right:{$0$}] at (-0.5,2.3){};
    \node[label=above right:{$0$}] at (4.3,0.4){};
    \node[label=above right:{$0$}] at (5,2.3){};
    \end{tikzpicture}
    \caption{The Poincar\'e homology sphere results from $(1, 1, 1)$-surgery on the Borromean rings. The meridians with surgery slopes $(0, 0, 0)$ realize a $3$-adjacency of the Poincar\'e homology sphere to $S^3$.}
    \label{fig:poincare}
\end{figure}

Next, consider the example in Figure \ref{fig:hopf chain}. An exercise in Kirby calculus shows that $(1/2, 1, 1/2)$-surgery yields $-L(3, 1)$ and induced surgery on any proper sublink yields $S^3$.  This surgery description will demonstrate that $-L(3,1)$ is 3-adjacent to $S^3$. Thus, Hopf links can and do appear in the characterization of rational adjacencies. 
We define the family of \emph{Hopf-Brunnian} links as follows. An $n$-component link is Hopf-Brunnian if all $n-1$ component sublinks are split unions of Hopf links and unknots. We prove:

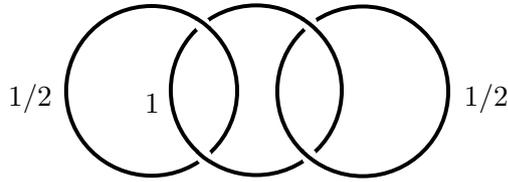
\begin{figure}
    \centering
    \begin{tikzpicture}

    \node[anchor=south west,inner sep=0] at (0,0) {\includegraphics[width=2in]{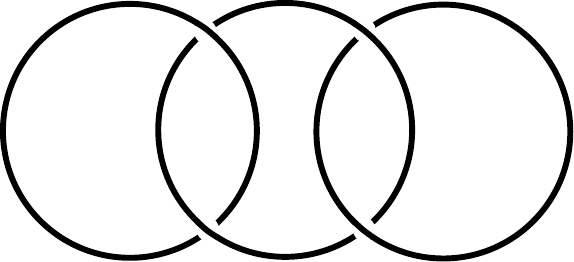}};

    \node[label=above right:{$1/2$}] at (-1,0.6){};
    \node[label=above right:{$1$}] at (0.8,0.6){};
    \node[label=above right:{$1/2$}] at (5,0.6){};
    \end{tikzpicture}
    \caption{This three-component Hopf-Brunnian link $J$ admits a $(1/2, 1, 1/2)$-rational surgery to $-L(3, 1)$. The core of $J$ after surgery is a link $L$ realizing the 3-adjacency of $-L(3,1)$ to $S^3$.}
    \label{fig:hopf chain}
\end{figure}

\begin{theorem}
\label{hopf-brunnian n}
The triple $(Y, L, \alpha)$ realizes an $n$-adjacency to $S^3$ with surgery core $J$ if and only if $J$ is Hopf-Brunnian, the dual surgery slopes of $J$ are of the form $1/k_i$, $k_i\in \mathbb{Z}^*$, and any proper Hopf sublink of $J$ must have surgery slopes $\pm(1, 1/2)$ or $\pm(1/2, 1)$.
\end{theorem}

\begin{corollary}
\label{integral adjacency Brunnian}
The triple $(Y, L, p)$ realizes an integral $n$-adjacency to $S^3$ with surgery core $J$ if and only if $J$ is Brunnian and the dual surgery slopes of $J$ are $\pm 1$ (signs need not be consistent).
\end{corollary}

\begin{corollary}
\label{homology sphere}
If $Y$ is integrally $n$-adjacent to $S^3$ for $n\geq3$, or rationally $n$-adjacent to $S^3$ for $n\geq4$, then $Y$ is an integer homology sphere.
\end{corollary}

Note that in the case $n=2$, any $(1/k_1, 1/k_2), k_i \in \mathbb{Z}^*$, surgery along any link $J_1\cup J_2$ of unknotted components in $S^3$ will yield a manifold adjacent to the three-sphere. 
In order to prove Theorem~\ref{hopf-brunnian n}, we give a stronger characterization for self-adjacencies from the three-sphere to itself.

\begin{reptheorem}{prop:rational-self-adjacency}
The triple $(S^3, J, \alpha)$ realizes an $n$-adjacency to $S^3$ if and only if $J$ itself is a split union of Hopf links and unknots, all slopes $\alpha_i = 1/k_i$, where $k_i\in\mathbb{Z}^*$, and the surgery slopes of Hopf components are either $\pm(1, 1/2)$ or $\pm(1/2, 1)$. 
\end{reptheorem}

Note that the requirement that $J$ itself is a split union of Hopf links and unknots is stronger than requiring $J$ be Hopf-Brunnian. 

Using Proposition~\ref{prop:rational-self-adjacency}, we may now prove Theorem \ref{hopf-brunnian n}.  

\begin{proof}[Proof of Theorem \ref{hopf-brunnian n}]
In Proposition~\ref{prop:dual} below, we show that if $(Y,L, \alpha)$ realizes a rational $n$-adjacency to $S^3$, then the core of the surgery is an $n$-component link $J$ in $S^3$ and performing $\alpha$-framed surgery on $\cup_{i \in I } L_i$, for any $I \subset \{1,\ldots,n\}$, yields the same result as performing surgery on $\cup_{i \in [n] - I} J_i$ with the corresponding dual slopes (see Section \ref{sec:dual} for more details).  In particular, surgery on every proper sublink of $J$ in $S^3$ gives back $S^3$.       
By Proposition~\ref{prop:rational-self-adjacency}, every proper sublink of $J$ is a split union of Hopf links and an unlink. The Hopf pairs have surgery coefficients $\pm(1, 1/2)$ and split unknotted components have surgery coefficients $1/k$, $k\neq0$. \end{proof}

Let us return to $n$-adjacency in knots. 
Using Theorem~\ref{hopf-brunnian n}, we are now able to recover Askitas-Kalfagianni's characterization of knots which are $n$-adjacent to the unknot \cite[Theorem 4.4]{AK}. 
\begin{theorem}[Askitas-Kalfagianni] 
\label{ak again}
Let $K$ be $n$-adjacent to the unknot for $n \geq 3$.  Then $K$ is the realization of a Brunnian-Suzuki $n$-graph.  
\end{theorem}

\begin{proof}
Let $n\geq3$, and let $K$ be $n$-adjacent to the unknot $U$. The collection of unknotting arcs lifts to a strongly invertible link $L=L_1\cup \ldots \cup L_n$ in the double cover of $S^3$ branched over $K$, which we will call $Y$. Likewise, there is a corresponding collection of `knotting' arcs $\gamma_1, \ldots, \gamma_n$ from $U$ to $K$, which lifts to a strongly invertible link $J$ in $S^3$.  The Montesinos trick \cite{Montesinos} provides a half-integral multi-slope $\beta$ on $J$ such that $\beta_i$-surgery on $J_i$ corresponds to the associated crossing change downstairs on $\gamma_i$.

Applying crossing changes associated to any subset of crossing arcs $\cup_{i\in I}\gamma_i$ along $U$ yields the same result as applying the complementary $[n]-I$ crossing changes in $K$.
By the $n$-adjacency of $K$, this produces the unknot for any proper subset $I\subset \{1,\cdots, n\}$.
 At the level of the branched double cover, we see that surgery on every proper sublink of $J$ produces $S^3$. (See Proposition~\ref{prop:dual} below.) We now apply Proposition~\ref{prop:rational-self-adjacency} to the $(n-1)$-component sublinks of $J$.  Since $n\geq 3$ and none of the surgery slopes are $\pm 1$ (because they are all half-integral), all of the pairwise linking numbers of $J$ must be zero and the proper sublinks are unlinks, \emph{i.e.} $J$ is Brunnian, and each $\beta_i$ is $\pm 1/2$. 

Consider now the union of the arcs $\gamma_1,\ldots, \gamma_n$ together with the unknot $U$. Again by the Montesinos trick \cite{Montesinos}, an arc with weight $(w, z)$ in the terminology of \cite[Section 3]{AK} realizes a surgery in the branched double cover with slope $w+\frac{z}{2}$. The theorem will follow from the following claim.
\end{proof}
\begin{claim}
The graph $G = U\cup \bigcup_{i=1}^n\gamma_i$ is a Brunnian Suzuki $n$-graph.
\end{claim}
\begin{proof}[Proof of claim]
    Notice that no pair of arcs $\gamma_i$ and $\gamma_j$ have endpoints interleaved along the unknot because they would lift to a link $J_i\cup J_j$ of nonzero linking number, a contradiction. 
    This means that $G$ is \emph{admissible}, in the terminology of \cite[Definition 3.2]{AK}.  
    For any proper subset $I\subset \{1, \ldots, n\}$ of components of $J_1\cup\ldots \cup J_n$, the quotient under $\tau$ is a proper subset of the $\gamma$ arcs. Because each subset of $J$ is an unlink, and there is a unique strong inversion $\tau$ on the unlink \cite{KT:Split}, these descend to arcs embedded disjointly in the spanning disk for the unknot. 
    Morever, because the surgery slopes on each $J_i$ are $\pm 1/2$, these descend to weighted arcs of the form $(0, \pm 1)$. Thus, subgraphs of $G$ with $n-1$ arcs are standard.
    The graph $G$ is therefore a Brunnian-Suzuki $n$-graph, as in \cite[Definition 3.4]{AK}.
\end{proof}

Note that Theorem \ref{ak again} is the key ingredient in the claimed vanishing of the Vassiliev invariants of a knot $n$-adjacent to the unknot, as mentioned above.

In analogy with the work of Askitas-Kalfagianni, we are also able to apply Theorem~\ref{hopf-brunnian n} more generally to prove a vanishing result for finite-type invariants of homology spheres.  (For a quick survey on finite-type invariants, see \cite{LinSurvey}.)    
\begin{corollary}
Let $Y$ be a homology sphere which is integrally $n$-adjacent to $S^3$.  Then all finite-type invariants of order less than $2n-4$ vanish.  
\end{corollary}
\begin{proof}
By Corollary~\ref{integral adjacency Brunnian}, $Y$ is obtained by surgery on an $n$-component Brunnian link.  The required vanishing result is given by \cite[Theorem 1.1]{Meilhan}.
\end{proof}

\section{Dehn surgery}

\subsection{The dual perspective}
\label{sec:dual}
Let $Y$ be a closed, oriented three-manifold, and let $L=L_1\cup...\cup L_n$ denote a link in $Y$. 
Let $\alpha=(\alpha_1, \ldots, \alpha_n)$ denote a multi-slope on $L$. 
The notation $Y_{\alpha}(L)$ denotes the three-manifold obtained by performing $\alpha_i$ Dehn surgery along $L_i$ for all $i=1, \ldots, n$. 

\begin{definition}
Consider the triple $(Y, L, \alpha)$, where $Y$ is a closed, oriented 3-manifold, $L=L_1\cup \ldots \cup L_n$ is a link in $Y$, and $\alpha=(\alpha_1, \ldots, \alpha_n)$ is a multi-slope on $L$. Let $Z$ be a closed, oriented three-manifold. 
If $Y_{\alpha_I}(L_I) = Z$ for \emph{any} nonempty subset $I$ of $\{1, ..., n\}$, then $(Y, L, \alpha)$ \emph{realizes} an $n$-adjacency to $Z$. We say that $Y$ is \emph{integrally} $n$-adjacent to $Z$ if the multi-slopes are integral and \emph{rationally} adjacent to $Z$ otherwise.
\end{definition}

Notice that $n$-adjacency is not a symmetric relation.

In order to circumvent the difficulty in describing surgeries in arbitrary three-manifolds, we take the following perspective.

Suppose that $L$ is an $n$-component link in a three-manifold $Y$ with a surgery to $S^3$ along the multi-slope $\alpha$.  Let $J$ be the core of the surgery in $S^3$, and write $J = J_1 \cup \ldots \cup J_n$.  Then, we will associate to $J$ rational numbers $r=(r_1,\ldots, r_n)$ which describe how to ``undo'' the surgery performed on each $J_i$.  To be more precise, in the exterior of $L$, we have two slopes $\eta_i$ and $\alpha_i$ on the boundary torus coming from $L_i$:  $\eta_i$ is the meridian of $L_i$ and Dehn filling along $\alpha_i$ corresponds to the non-trivial surgery we are going to do to get to $S^3$.  Viewing $J$ as a link in $S^3$, we still can view these slopes in the boundary of a neighborhood of $J_i$.  Express $\eta_i = p_i \mu_i + q_i \lambda_i$, where $\mu_i, \lambda_i$ are the meridian and longitudes for $J_i$ as a knot in $S^3$.  (Note that $\mu_i = \alpha_i$.)  We say that the slopes $\eta_i$ and $\alpha_i$ are dual to each other.  We calculate these in the following order: 
\begin{enumerate}
\item First, perform surgery on {\em all} components of $L$ to get $J$ in $S^3$, not just some of the components (which still produces $S^3$ if $L$ is realizing an $n$-adjacency).
\item Then, identify the slopes $\eta_i$ with $r_i=p_i/q_i$.      
\end{enumerate}   
Note that the adjacency is integral if and only if all $p_i/q_i$ are integral.  This is because $\Delta(\eta_i,\alpha_i) = \Delta(\mu_i, p_i \mu_i + q_i \lambda_i) = |q_i|$.  In general, when discussing the dual link of a surgery to $S^3$, we will assume that it naturally inherits these rational surgery slopes in $S^3$ as above.

Now, the data $(S^3, J, r)$ in $S^3$ actually recovers $(Y,L,\alpha)$.  Performing surgery on all components of $J$ gives $Y$ and by construction $L$ is the core of the surgery on $J$ while the meridian of each $J_i$ becomes the slope $\alpha_i$.  But, we can also recover sublinks in the following way.  If we look at a sublink $J'$ of $J$, without loss of generality, $J_1 \cup \ldots \cup J_k$, then surgery on $J'$ produces the same manifold as surgery in $Y$ on $L_{k+1} \cup \ldots \cup L_n$.  And further, the core of surgery on $J'$, a $k$-component link, is exactly the {\em image} of $L_1 \cup \ldots \cup L_k$ in the surgery on $L_{k+1} \cup \ldots \cup L_n$.  

Now we return to the case that a rational surgery on $L$ is realizing an $n$-adjacency from $Y$ to $S^3$.  Build $(S^3, J,r)$ as discussed.  The above paragraph can be reinterpreted as saying that doing the corresponding surgery on every proper sublink of $J$ gives $S^3$.  For the benefit of the reader, we summarize this discussion with the following proposition.

\begin{proposition}\label{prop:dual}
Let $\alpha$ be a multi-slope on an $n$-component link $L$ in $Y$.  Then $(Y,L,\alpha)$ realizes an $n$-adjacency to $S^3$ if and only if there exists a multi-slope $\beta$ on a link $J$ in $S^3$ such that: 
\begin{enumerate}
\item $S^3_\beta(J) = Y$;
\item $L$ is the core of the surgery on $J$;
\item each $\alpha_i$ is the dual slope to $\beta_i$; 
\item surgery on every proper sublink of $J$ yields $S^3$.
\end{enumerate}
Furthermore, the adjacency is integral if and only if $\beta$ is integral.  
\end{proposition}

\subsection{The linking of the dual curves}
  In light of the dual perspective from Proposition~\ref{prop:dual}, we want to understand the effects of surgery on links in $S^3$ whose sublinks also surger to $S^3$.  The next lemma allows us to constrain the linking numbers and surgery coefficients for the dual link in $S^3$ arising from an $n$-adjacency.  

\begin{lemma}
\label{linking number n=2}
Suppose that $(S^3, J,\alpha)$ realizes a 2-adjacency to $S^3$.  Then either the linking number of $J$ is zero, or the linking number is $\pm 1$ and the surgery coefficients $\alpha_i$ are $\pm (1,1/2)$ or $\pm (1/2, 1)$.   
\end{lemma}
\begin{proof}
Since surgery on each individual component of $J = J_1 \cup J_2$ produces $S^3$, an integer homology sphere, the surgery coefficient for $J_i$ is of the form $1/q_i$.  The linking matrix for the  surgery presentation on $J$ then gives 
\[
1 = \left |det \begin{pmatrix} 1 & q_2 \ell \\ q_1 \ell & 1 \end{pmatrix} \right | = |q_1 q_2 \ell^2 - 1|,
\]
where $\ell$ is the linking number of $J_1$ and $J_2$ (say after choosing orientations of each component).  
Since $q_1, q_2 \neq 0$, we see that $|\ell| = 0$ or $1$.  If $|\ell| = 1$, then we must have that $(q_1,q_2) = \pm (1,2)$ or $\pm (2, 1)$, as desired.  
\end{proof}

\begin{proposition}
\label{prop:linking}
Suppose that $J$ is an $n$-component link in $S^3$, with $n \geq 3$, and $\alpha$ is a multi-slope on $J$.  If surgery on every proper sublink of $J$ produces   $S^3$, then all pairwise linking numbers are 0 or $\pm 1$.  If $J_1$ and $J_2$ are a pair of components  with $|\ell k(J_1,J_2)| = 1$, then the slopes are $\pm (1,1/2)$ or $\pm (1/2, 1)$.  If $n = 3$ and $S^3_\alpha(J)$ is an integer homology sphere or $n \geq 4$, then each of $J_1$ and $J_2$ has linking number zero with all other components.
\end{proposition}
\begin{proof}
Since $n \geq 3$, every two-component sublink of $J$ with induced multi-slope from $\alpha$ provides a 2-adjacency from $S^3$ to itself. 
Therefore, by Lemma~\ref{linking number n=2}, the pairwise linking numbers are 0 or $\pm 1$ and for the 2-component sublinks with linking number having absolute value 1, the surgery coefficients are $\pm (1,1/2)$ or $\pm (1/2,1)$. 

Now suppose that $S^3_\alpha(J)$ is an integer homology sphere.  It remains to consider the pairwise linking of $J_1, J_2$ with the other components.  Let $J_3$ be another component.  Then, we know that the associated surgery on $J_1 \cup J_2 \cup J_3$ produces $S^3_\alpha(J)$ if $n = 3$ and $S^3$ if $n > 3$.  Either way, the result is an integer homology sphere.  Orient $J_1, J_2$ such that the pairwise linking is 1, fix an orientation on $J_3$ and let $\ell_1, \ell_2$ be the linking numbers of $J_3$ with $J_1, J_2$ respectively.  Without loss of generality, the surgery coefficients on $J_1$ and $J_2$ are 1 and $1/2$ respectively.  (Otherwise, rearrange the order of the components and/or mirror $J$ and reverse the signs of $\alpha$.)  Suppose for contradiction that $\ell_1, \ell_2$ are not both zero.     

The first case is that $\ell_1  \neq 0$.  In this case, by applying the first part of the proposition to the pair $(J_1, J_3)$, we see that $\ell_1 =1$ and the surgery coefficient for $J_3$ must be $1/2$.  Applying the first part of the proposition to the pair $(J_2, J_3)$, we see that $J_2$ and $J_3$ have linking number zero, since the pair of surgery coefficients is not $\pm (1/2,1)$ or $\pm (1,1/2)$.  In this case, the linking matrix for the 3-component surgery description computes the order of $H_1$ of the surgery on $J_1 \cup J_2 \cup J_3$ to be: 
\[
1 = \left |det \begin{pmatrix} 1 & 2 &  2 \\ 1 & 1 & 0 \\ 1 & 0 & 1 \end{pmatrix} \right | = 3,
\]
a contradiction.  

The other case is that $\ell_1=0$, and so $\ell_2 = 1$. Now we see the surgery coefficients are $1$ for $J_1$ and $J_3$, and $1/2$ for $J_2$.  Then, we compute again
\[
1 = \left | det \begin{pmatrix} 1 & 2 & 0 \\ 1 & 1 & 1 \\ 0 & 2 & 1 \end{pmatrix} \right | = 3,  
\]
another contradiction.  This completes the proof.  
\end{proof}

\section{Self-adjacencies from $S^3$}
In this section, we constrain the self-adjacencies from $S^3$ to itself.  
As a warm-up, we begin with a special case.  

\begin{proposition}\label{prop:brunnian-self-adjacency}
Suppose $(S^3,J,\alpha)$ realizes an $n$-adjacency to $S^3$ and the pairwise linking numbers of $J$ vanish.  Then $J$ is the unlink and $\alpha_i = 1/k_i$ for $k_i \in \mathbb{Z}^*$ for all $i$.  
\end{proposition}
\begin{proof}
First, it is clear that $\alpha_i = 1/k_i$ for each $i$ by a homological computation.  We proceed by induction to show that $J$ is the unlink.  The case of $n = 1$ is handled by the knot complement theorem \cite{GordonLuecke}.  Next we do the case of a 2-component link, $J_1 \cup J_2$.    From the $n = 1$ case, each component is unknotted. Since $1/k_2$-surgery on $J_2$ is $S^3$, the image of $J_1$ in $1/k_2$-surgery on $J_2$ has a surgery to $S^3$.  Hence, the image of $J_1$ is also unknotted.  In particular, the component $J_2$ in the complement of $J_1$ in $S^3$ is a knot in a solid torus which has a non-trivial solid torus surgery.  By \cite{Gabai}, $J_1$ is contained in a ball or is a braid in the solid torus, so has non-zero winding number. However, because $\ell k(J_1,J_2)=0$, it must be the case that $J_1$ is contained in a ball in the complement of $J_2$ in $S^3$, meaning $J_1$ and $J_2$  are unlinked.  This completes the proof for $n = 2$ components.  

For the inductive step, our hypothesis is that $J$ is a Brunnian link and then we will deduce that $J$ is in fact an unlink. This can be found in \cite[Proposition 4.1]{GLLM}, but for self-containedness, we present an elementary proof that does not rely on Heegaard Floer homology.  Suppose the result is true for $(n-1)$-component links and that $(S^3,J,\alpha)$ realizes an $n$-adjacency to $S^3$.  Then $J$ is Brunnian, $J_1 \cup \ldots \cup J_{n-1}$ is an unlink, and so $(1/k_1,\ldots, 1/k_{n-1})$-surgery on $J_1 \cup \ldots \cup J_{n-1}$ gives $S^3$.  Thus the image of $J_{n}$ must be unknotted after this surgery.  Hence, we see that $(1/k_1, \ldots, 1/k_{n-1}, 1/m)$-surgery on $J$ gives $S^3$ for arbitrary $m$. For the sake of concreteness, fix $m=5$. 

We claim that the image of $J_1 \cup \ldots \cup J_{n-1}$ after performing $1/5$-surgery on $J_n$, denoted $K_1 \cup \ldots \cup K_{n-1}$, yields an $(n-1)$-component Brunnian link. To see that the $(n-1)$-component image link is Brunnian, note that all $(n-1)$-component sublinks of $J$ are unlinks by our inductive hypothesis. In particular, $J_1\cup\ldots\cup J_{n-2}\cup J_n$ is an unlink, hence $1/5$-surgery along $J_n$ shows that $K_1 \cup\ldots\cup K_{n-2}$ is an unlink.  A similar argument applies to the other $n-2$-component sublinks of $K_1\cup \cdots \cup K_{n-1}$.

Thus, the link $K_1 \cup \cdots \cup K_{n-1}$ is a Brunnian link with a surgery to $S^3$, and so is an unlink by induction.  In other words, $1/5$-surgery on $J_n$ in the exterior of $J_1 \cup \ldots \cup J_{n-1}$ produces a reducible 3-manifold (the exterior of an $(n-1)$-component unlink).  
Yet, the distance between the trivial slope $\infty$ and $1/5$ is $5$. A theorem of Gordon and Litherland 
\cite[Theorem 1.1]{GL} implies that for any pair of reducible Dehn fillings on an irreducible manifold, the slopes have distance at most four. This is a contradiction.  
This implies that the exterior of $J$ must be reducible.  Hence, $J$ is split.  
However, a split Brunnian link is an unlink.
\end{proof}

We now build on the previous proposition to complete our characterization of the self-adjacencies of $S^3$ promised in the introduction.  

\begin{proposition}
\label{prop:rational-self-adjacency}
The triple $(S^3, J, \alpha)$ realizes an $n$-adjacency to $S^3$ if and only if $J$ itself is a split union of Hopf links and unknots, all slopes $\alpha_i = 1/k_i$, where $k_i\in\mathbb{Z}^*$, and the surgery slopes of Hopf components are either $\pm(1, 1/2)$ or $\pm(1/2, 1)$. 
\end{proposition}

\begin{proof}
As in the proof of Proposition~\ref{prop:brunnian-self-adjacency}, the knot complement theorem implies that the components are unknotted and of course the surgery coefficients are of the form $1/k_i$.  

We begin with the case of $n = 2$. Consider $J_1\cup J_2$. The case that $\ell k(J_1, J_2)=0$ follows from Proposition \ref{prop:brunnian-self-adjacency}. By Lemma \ref{linking number n=2}, we assume $\ell k(J_1, J_2)=1$ and the surgery coefficients are $\pm (1,1/2)$. Note that $J_1$ is unknotted in $1/2$-surgery on $J_2$, so $J_2$ can again be viewed as a knot in the solid torus with a solid torus surgery. 
 Therefore, by \cite{Gabai}, $J_2$ is a braid in the complement of $J_1$ and the winding number is the linking number of $J_1$ and $J_2$.  The only winding number 1 braid in the solid torus is the core.  We see that $J_1 \cup J_2$ is a Hopf link.  

Next, we handle the case of $n = 3$.  If the pairwise linking numbers are zero, we appeal to Proposition~\ref{prop:brunnian-self-adjacency}.  So, assume some pair of components $J_1, J_2$ have nonzero linking number.  By Proposition~\ref{prop:linking}, $J_1, J_2$ have surgery coefficients $\pm (1,1/2)$ and linking number 1.  Further, by Proposition~\ref{prop:linking}, we have that $J_3$ is algebraically split from $J_1$ and $J_2$. By the $n = 2$ case of the proof, $J_1 \cup J_2$ must form a Hopf link.  We also have that $J_3$ is an unknot that is geometrically split from $J_1$ and $J_2$ individually, but possibly not split from the link $J_1 \cup J_2$.

It remains to show that $J_3$ is in fact split from $J_1 \cup J_2$. 
Since $J_1 \cup J_3$ is an unlink, trivial surgery on $J_2$ produces the 2-component unlink $J_1 \cup J_3$. 
Also, the image of  $J_1 \cup J_3$ under $1/2$-surgery on $J_2$ is a 2-component unlink because this image is a 2-component link, all of whose induced surgeries give $S^3$.  
We now appeal to \cite[Corollary 2.4.7]{CGLS}.  This states that if an irreducible three-manifold admits two reducible Dehn fillings along slopes of distance at least two on a torus boundary component, one of the filled manifolds contains a lens space summand.  However, a link complement in $S^3$ cannot contain a lens space summand.  As the surgered manifolds are link complements in $S^3$, the exterior of $J$ is  reducible, and so $J$ is a split link. 

Now, we complete the induction using a similar strategy.  Suppose that $J$ has $n$-components. By assumption, all proper sublinks are split unions of unlinks and Hopf links.  If $J$ has pairwise linking numbers all zero, we can again apply Proposition~\ref{prop:brunnian-self-adjacency}.  Therefore, up to reordering of the components, we have at least one two-component sublink $J_1 \cup J_2$ which is a Hopf link. Up to mirroring and reordering $J_1$ and $J_2$, the surgery coefficient on $J_1$ is $1/2$. Trivial surgery on $J_2$ produces a split link as does $1/2$-surgery, since the resulting link is an $(n-1)$-component link satisfying the same hypotheses of the theorem. Therefore, by applealing again to \cite[Corollary 2.4.7]{CGLS}, we get that the complement of $J$ is reducible, so $J$ is split.  Since all $(n-1)$-component sublinks are split unions of Hopf links and unknots, and because $J$ itself is split, we now have that $J$ is a split union of Hopf links and unknots.  
\end{proof}

\subsection*{Acknowledgements}  TL was supported by NSF DMS--2105469.  AHM was supported by NSF Grant DMS--2204148 and The Thomas F. and Kate Miller Jeffress Memorial Trust, Bank
of America, Trustee.

\bibliographystyle{alpha}
\bibliography{biblio}

\end{document}